\documentclass[pdflatex,sn-mathphys-num]{sn-jnl}% Math and Physical Sciences Numbered Reference Style

\usepackage{graphicx}%
\usepackage{multirow}%
\usepackage{amsmath,amssymb,amsfonts}%
\usepackage{amsthm}%
\usepackage{mathrsfs}%
\usepackage[title]{appendix}%
\usepackage{xcolor}%
\usepackage{textcomp}%
\usepackage{manyfoot}%
\usepackage{booktabs}%
\usepackage{algorithm}%
\usepackage{algorithmicx}%
\usepackage{algpseudocode}%
\usepackage{listings}%
\theoremstyle{thmstyleone}%
\newtheorem{theorem}{Theorem}[section]% meant for sectionwise numbers
\newtheorem{proposition}[theorem]{Proposition}%
\newtheorem{lemma}[theorem]{Lemma}%

\theoremstyle{thmstyletwo}%
\newtheorem{remark}[theorem]{Remark}%

\theoremstyle{thmstylethree}%
\newtheorem{definition}[theorem]{Definition}%

\numberwithin{equation}{section}

\newcommand{\R}{\mathbb{R}}
\newcommand{\Z}{\mathbb{Z}}
\newcommand{\T}{\mathbb{T}}
\newcommand{\Torus}{\mathbb{T}^2}

\begin{document}

\title[An Anisotropic Onsager Criterion]
{An Anisotropic Onsager Criterion for the Two-Dimensional
Navier--Stokes Equations with Horizontal Viscosity}

%%=============================================================%%
%% GivenName	-> \fnm{Siyu}
%% FamilyName	-> \sur{Liang}
%%=============================================================%%

\author*[1]{\fnm{Siyu} \sur{Liang}}\email{liangsiyu1994@163.com}

\affil*[1]{\orgdiv{School of Mathematics and Statistics},
\orgname{Nanjing University of Science and Technology},
\orgaddress{\city{Nanjing}, \postcode{210094}, \country{China}}}


\abstract{\unboldmath We consider distributional weak solutions $u\in L^\infty_tL^2_x$ of the two-dimensional Navier--Stokes equations with horizontal viscosity on $\mathbb T^2$ or $\mathbb R^2$, without assuming $\partial_1u\in L^2_tL^2_x$. We show that the Lions integrability condition $u\in L^4_tL^4_x$ together with Onsager-critical regularity in the nondissipative direction, $u\in L^3_tB^{1/3,v}_{3,\infty}$ with the Besov regularity imposed only in the vertical direction, implies $\partial_1u\in L^2_tL^2_x$, with a quantitative bound on the horizontal dissipation. Thus the dissipative regularity is a consequence of the equation and need not be part of the definition of the solution. If moreover $u\in L^3_tB^{1/3,v}_{3,c_0}$, then $u$ belongs to the isotropic critical Onsager space $L^3_tB^{1/3}_{3,c_0}$, satisfies the energy equality, and is continuous in $L^2$. Here $c_0$ indicates that the corresponding dyadic Besov sequence tends to zero at high frequencies. The proof combines absorption of the horizontal energy flux into the dissipation with a one-dimensional commutator estimate for the vertical energy flux.
}



\maketitle

\section{Introduction}\label{sec:intro}

We consider the two-dimensional incompressible Navier--Stokes equations with viscosity
acting in the horizontal direction only,
\[
\partial_t u + \operatorname{div}(u\otimes u) + \nabla p - \nu\,\partial_1^2 u = 0,
\qquad \nabla\cdot u = 0,
\]
on the torus $\Torus$ or on the whole plane $\R^2$. Anisotropic dissipation arises in
geophysical fluid dynamics, where horizontal and vertical effective viscosities differ
greatly. See \cite{CDGG2000,CDGGbook,Paicu2005} for the anisotropic Navier--Stokes
equations in anisotropic Sobolev and Besov spaces. In particular, in our joint work with Zhang and Zhu~\cite{LZZ}, we established global
well-posedness for the two-dimensional equation above on both domains in the anisotropic
space $\widetilde H^{0,1}$, which requires one full vertical derivative in $L^2$. Since the
dissipation acts in one direction only, the equation is intermediate between the Navier--Stokes
equations and the Euler equations. We now ask how little regularity is needed for the
energy equality
\[
\tfrac12\|u(t)\|_{L^2}^2
+\nu\int_0^t\|\partial_1u(s)\|_{L^2}^2\,ds
=\tfrac12\|u_0\|_{L^2}^2
\]
to hold.

For the two limiting models the situation is as follows. On the Navier--Stokes side, the classical sufficient condition of Lions \cite{Lions1960} is $u\in L^4(0,T;L^4)$, and Shinbrot \cite{Shinbrot1974} extended it to $u\in L^p(0,T;L^q)$ with $2/p+2/q\le1$ and $q\ge4$. Further sufficient conditions can be found in \cite{BerselliChiodaroli2020,CheskidovLuo2020,Galdi2019}. On the Euler side,
the question is the content of Onsager's conjecture \cite{Onsager1949}, which singles out
$1/3$ as the critical regularity exponent. Energy conservation above this exponent was
proved by Constantin, E and Titi \cite{CET1994} for $u\in L^3(0,T;B^{\alpha}_{3,\infty})$
with $\alpha>1/3$ (see also \cite{Eyink1994,DuchonRobert2000}), and was extended to the
critical space $L^3(0,T;B^{1/3}_{3,c_0})$ by Cheskidov, Constantin, Friedlander and
Shvydkoy \cite{CCFS2008}. Below the critical exponent, weak solutions that fail to
conserve energy were constructed in three dimensions in \cite{Isett2018,BDSV2019} and in
two dimensions in \cite{GiriRadu2024}.

In two dimensions with full viscosity, every Leray--Hopf solution
\cite{Leray1934,Hopf1951} satisfies the energy equality.
On the torus the conclusion extends to distributional solutions in
$C([0,T];L^2)$, which are unique by the argument of \cite{FLRT2000} (see also
\cite{CheskidovLuo2023}) and hence coincide with the Leray solution. The endpoint is sharp, as Cheskidov and Luo \cite{CheskidovLuo2023} construct nonunique
solutions in $C([0,T];L^p)$ for every $p<2$. With horizontal viscosity only,
neither argument applies. The dissipation controls $\partial_1u$ but not $\partial_2u$, so
the Ladyzhenskaya estimate is unavailable, and the parabolic smoothing acts only in
horizontal frequencies. At the natural energy level of the equation, that is, for
$u\in L^\infty(0,T;L^2)$ with $\partial_1u\in L^2(0,T;L^2)$ assumed, Demmel and Wiedemann
\cite{DemmelWiedemann2026} recently proved the energy equality for $\Omega=\R^2$, by means
of a directional Riesz-transform estimate for the pressure and a renormalization argument
for the horizontal component. The periodic case is not treated there.
For the case $\Omega=\R^2$, our conclusion (i) of Theorem~\ref{thm:main} provides this energy regularity,
so their theorem yields the energy equality (Remark~\ref{rem:R2}).

In the present paper we start from distributional solutions in $L^\infty(0,T;L^2)$ only,
without assuming $\partial_1u\in L^2(0,T;L^2)$, and we determine conditions under which
this energy level is reached. This is the main point of the paper.
Theorem~\ref{thm:main} states that the Lions integrability together with
Onsager-critical regularity in the inviscid direction suffices. More precisely, if a weak
solution $u\in L^\infty(0,T;L^2)$ satisfies
$u\in L^4(0,T;L^4)\cap L^3(0,T;B^{1/3,v}_{3,\infty})$, with the Besov regularity measured
in $x_2$ alone, then $\partial_1u\in L^2(0,T;L^2)$ on both domains.
Under the $c_0$ refinement $u\in L^3(0,T;B^{1/3,v}_{3,c_0})$, the field $u$ lies
moreover in the isotropic critical Onsager space $L^3(0,T;B^{1/3}_{3,c_0})$ of
\cite{CCFS2008}, the energy equality holds, and $u\in C([0,T];L^2)$. Neither the isotropic
critical regularity nor the continuity in $L^2$ is assumed. The energy equality is proved
directly from the horizontal dissipation estimate and does not use the isotropic space.

The global well-posedness result of \cite{LZZ} assumes one full vertical derivative of the initial
datum in a Sobolev space. Here horizontal dissipation and the energy equality are
established for an existing weak solution under assumptions on the solution itself, with
one third of a vertical derivative in a time-integrated Besov space.

The proof starts from an energy balance for the field truncated at horizontal frequency $N$
and vertical frequency $Q$. The horizontal flux is absorbed by the horizontal dissipation,
with the commutator controlled by $\|u\|_{L^4}^2$, while the vertical flux is handled by a
one-dimensional commutator estimate in the spirit of \cite{CET1994,CCFS2008}, for which
$1/3$ is the critical exponent. Letting $N\to\infty$ and then $Q\to\infty$ yields the
horizontal dissipation estimate, and the energy equality then follows from the vertically
truncated balance alone.

The paper is organized as follows.
Section~\ref{sec:statement} states the main result and Section~\ref{sec:commutator}
collects the one-dimensional Littlewood--Paley and commutator estimates.
Section~\ref{sec:apriori-energy} proves the horizontal dissipation estimate, the isotropic
critical regularity, and the energy equality. Appendix~\ref{app:A} derives the doubly truncated energy balance.

\section{Statement of the result}\label{sec:statement}

We work on $\Omega = \Omega_1\times\Omega_2$, where $\Omega_1$ is the horizontal (viscous)
factor and $\Omega_2$ is the vertical (inviscid) factor, in either of the two cases
\begin{equation}
\begin{split}
\Omega = \Torus = \T\times\T,\quad \T:=\R/2\pi\Z, &\quad (\Omega_1=\Omega_2=\T),\\
\text{or}\qquad\Omega = \R^2 = \R\times\R, &\quad (\Omega_1=\Omega_2=\R).
\end{split}
\label{eq:domains}
\end{equation}
We write $x=(x_1,x_2)\in\Omega$ and consider
\begin{equation}
\partial_t u + \operatorname{div}(u\otimes u) + \nabla p - \nu \partial_1^2 u = 0,
\qquad \nabla\cdot u = 0
\label{eq:NS}
\end{equation}
on $(0,T)\times\Omega$, where $\nu>0$.

Throughout the paper, $L^2_\sigma(\Omega)$ denotes the space of divergence-free vector
fields in $L^2(\Omega;\R^2)$. We abbreviate mixed space-time norms by subscripts, writing
$L^p_tX$ for $L^p(0,T;X(\Omega))$ with $1\le p\le\infty$, and $L^p_tL^q_x$ for
$L^p(0,T;L^q(\Omega))$. We also write $L^2((0,T)\times\Omega)$ for $L^2(0,T;L^2(\Omega))$.
When the spatial domain is clear, we omit it from the notation for spaces and norms.
The test space $\mathscr S(\Omega)$ is $C^\infty(\Torus)$ when $\Omega=\Torus$ and the
Schwartz space $\mathscr S(\R^2)$ when $\Omega=\R^2$, and $\mathscr S'(\Omega)$ denotes its
dual, the space of distributions on $\Torus$ or of tempered distributions on $\R^2$.

Throughout, $C>0$ denotes a constant depending only on the fixed cutoffs $\chi,\varphi$
and the profile $\theta$ in Appendix~\ref{app:A}. It may vary from one occurrence to
another and is independent of $\nu,T,u,N,Q$. We call such constants \emph{universal}
and write $a\lesssim b$ if $a\le Cb$, and $a\sim b$ if $a\lesssim b\lesssim a$.

We use one-dimensional Littlewood--Paley decompositions in each of the two variables,
following the notation of \cite[Chapter~2]{BCD}. Let
$\mathcal B:=\{\xi\in\R:|\xi|<4/3\}$ and $\mathcal C:=\{\xi\in\R:3/4\le|\xi|\le8/3\}$. By
\cite[Proposition~2.10]{BCD} there exist even functions $\chi\in C_c^\infty(\mathcal B)$
and $\varphi\in C_c^\infty(\mathcal C)$, with values in $[0,1]$, such that
\begin{equation}
\chi(\xi)+\sum_{q\ge0}\varphi(2^{-q}\xi)=1
\qquad\text{for every } \xi\in\R .
\label{eq:partition}
\end{equation}
We fix such a pair $(\chi,\varphi)$ once and for all.

The dual variable $\xi$ runs over $\R$ when $\Omega_2=\R$ and over $\Z$ when
$\Omega_2=\T$, and since \eqref{eq:partition} holds on all of $\R$, it holds in particular
on $\Z$. For a symbol $a\in C_c^\infty(\R)$ and $f\in\mathscr S(\Omega)$, set
\[
a(D_2)f(x_1,x_2)=
\begin{cases}
\displaystyle\frac1{2\pi}\int_\R a(\xi)\,\widehat f(x_1,\xi)\,e^{\,ix_2\xi}\,d\xi,
& \Omega_2=\R,\\[8pt]
\displaystyle\sum_{n\in\Z} a(n)\,\widehat f(x_1,n)\,e^{\,inx_2},
& \Omega_2=\T,
\end{cases}
\]
where we use the Fourier transform and Fourier coefficients
\[
\begin{aligned}
\widehat f(x_1,\xi)
&:=\int_\R f(x_1,y)e^{-iy\xi}\,dy,
&& \Omega_2=\R,\quad \xi\in\R,\\[4pt]
\widehat f(x_1,n)
&:=\frac1{2\pi}\int_{-\pi}^{\pi}f(x_1,y)e^{-iny}\,dy,
&& \Omega_2=\T,\quad n\in\Z.
\end{aligned}
\]
The corresponding inversion formulas are
\[
f(x_1,x_2)=
\begin{cases}
\displaystyle\frac1{2\pi}\int_\R\widehat f(x_1,\xi)e^{ix_2\xi}\,d\xi,
& \Omega_2=\R,\\[8pt]
\displaystyle\sum_{n\in\Z}\widehat f(x_1,n)e^{inx_2},
& \Omega_2=\T.
\end{cases}
\]
The operator $a(D_2)$ maps $\mathscr S(\Omega)$ continuously into itself, and it is symmetric
when $a$ is even. For the even symbols used below, $a(D_2)$ therefore extends to
$\mathscr S'(\Omega)$ by duality, $\langle a(D_2)f,g\rangle:=\langle f,a(D_2)g\rangle$ for
$f\in\mathscr S'(\Omega)$ and $g\in\mathscr S(\Omega)$. With this reading,
Definition~\ref{def:blocks} applies without change to both domains. The multiplier
$a(D_1)$ is defined similarly in the horizontal variable.

\begin{definition}\label{def:blocks}
Let $f\in\mathscr S'(\Omega)$. Following \cite[Definition~2.11]{BCD}, set
\[
\begin{aligned}
\Delta_{-1}^h f &= \chi(D_1)f, &\qquad
\Delta_q^h f &= \varphi(2^{-q}D_1)f \quad(q\ge0),\\
\Delta_{-1}^v f &= \chi(D_2)f, &
\Delta_q^v f &= \varphi(2^{-q}D_2)f \quad(q\ge0),
\end{aligned}
\]
with $\Delta_q^h f=\Delta_q^v f=0$ for $q\le-2$, and define
\[
\begin{aligned}
S_N^h f &= \sum_{q\le N}\Delta_q^h f = \chi(2^{-N-1}D_1)f,\qquad N\ge-1,\\
S_Q^v f &= \sum_{q\le Q}\Delta_q^v f = \chi(2^{-Q-1}D_2)f,\qquad Q\ge-1.
\end{aligned}
\]
\end{definition}

Throughout, $q$ and $j$ denote dyadic block indices, while $N$ denotes a horizontal
truncation level and $Q$ a vertical one. When a statement holds for either truncation,
we write $M$ for its level. All dyadic indices and truncation levels in this paper are integers.

\begin{definition}\label{def:besov}
For $s\in\R$ and $1\le p\le\infty$, let $B^{s,v}_{p,\infty}(\Omega)$ be the space of
$f\in\mathscr S'(\Omega)$ with
\begin{equation}
\|f\|_{B^{s,v}_{p,\infty}} := \sup_{q\ge -1} 2^{qs}\|\Delta_q^v f\|_{L^p(\Omega)}
<\infty ,
\label{eq:besov}
\end{equation}
and let $B^{s,v}_{p,c_0}(\Omega)$ be the subspace of those $f\in B^{s,v}_{p,\infty}(\Omega)$
for which
\begin{equation}
\lim_{q\to\infty} 2^{qs}\|\Delta_q^v f\|_{L^p(\Omega)} = 0.
\label{eq:czero}
\end{equation}
\end{definition}

The superscript $v$ indicates that the decomposition is performed only in the vertical
variable. The $c_0$ refinement of the critical Besov space is that of \cite{CCFS2008}.

The isotropic decomposition entering part (ii) of Theorem~\ref{thm:main} rests on
two-dimensional symbols. For $b\in C_c^\infty(\R^2)$ and $f\in\mathscr S(\Omega)$, the
operator $b(D)f$ is defined by the formulas above with $\R$ replaced by $\R^2$ and $\Z$ by
$\Z^2$, and it extends to $\mathscr S'(\Omega)$ by duality when $b$ is even. The two
symbols below are of this kind. Both are even and compactly supported, and both are
smooth: $\varphi$ vanishes on $\{|\xi|<3/4\}$, so $\varphi(2^{-q}|\cdot|)$ vanishes near
the origin and is smooth away from it, while $\chi\equiv1$ on $\{|\xi|<3/4\}$ by
\eqref{eq:partition}, so $\chi(|\cdot|)$ is constant near the origin.

\begin{definition}\label{def:isoblocks}
For $f\in\mathscr S'(\Omega)$, set
\[
\Delta_{-1}f=\chi(|D|)f,\qquad
\Delta_qf=\varphi(2^{-q}|D|)f\quad(q\ge0),\qquad
\Delta_qf=0\quad(q\le-2),
\]
\[
S_Qf=\sum_{q\le Q}\Delta_qf=\chi(2^{-Q-1}|D|)f,\qquad Q\ge-1.
\]
\end{definition}

\begin{definition}\label{def:isobesov}
For $s\in\R$ and $1\le p\le\infty$, let $B^s_{p,\infty}(\Omega)$ be the space of
$f\in\mathscr S'(\Omega)$ with
\[
\|f\|_{B^s_{p,\infty}}:=\sup_{q\ge-1}2^{qs}\|\Delta_qf\|_{L^p(\Omega)}<\infty,
\]
and let $B^s_{p,c_0}(\Omega)$ be the subspace of those $f\in B^s_{p,\infty}(\Omega)$ for which
\[
\lim_{q\to\infty}2^{qs}\|\Delta_qf\|_{L^p(\Omega)}=0.
\]
\end{definition}

\begin{samepage}
\begin{remark}\label{rem:BCD}
Our cutoffs $\chi,\varphi:\R\to[0,1]$ are even functions, obtained from
\cite[Proposition~2.10]{BCD} with $d=1$. In the isotropic decomposition, the radial
functions $\xi\mapsto\chi(|\xi|)$ and $\xi\mapsto\varphi(|\xi|)$, with $\xi\in\R^2$,
play the role of the two-dimensional cutoffs denoted by $\chi,\varphi$ in \cite{BCD}.

Our truncations sum over $k\le j$, whereas \cite{BCD} defines
$S_j=\sum_{k<j}\Delta_k$. Thus our $S_N^h$ and $S_Q^v$ correspond to $S_{N+1}$ and
$S_{Q+1}$ of \cite{BCD} applied in the respective coordinate directions, and our
isotropic $S_Q$ corresponds to their $S_{Q+1}$.
\end{remark}
\end{samepage}

\begin{definition}\label{def:weak}
Let $u_0\in L^2_\sigma(\Omega)$. A \emph{weak solution} of \eqref{eq:NS} on $[0,T]$ with
initial datum $u_0$ is a vector field $u\in L^\infty(0,T;L^2_\sigma(\Omega))$ such that,
for every divergence-free $\Phi\in C_c^\infty([0,T)\times\Omega;\R^2)$,
\begin{equation}
\int_0^T\!\!\int_\Omega \Bigl( u\cdot\partial_t\Phi + u\otimes u:\nabla\Phi
+ \nu\, u\cdot\partial_1^2\Phi \Bigr)\,dx\,dt
+\int_\Omega u_0\cdot\Phi(0,\cdot)\,dx = 0.
\label{eq:weakform}
\end{equation}
\end{definition}

\begin{remark}\label{rem:weakcont}
Every weak solution has a representative, again denoted $u$, which is weakly continuous
from $[0,T]$ into $L^2_\sigma(\Omega)$ and satisfies
\begin{equation}
u(0)=u_0,
\qquad
\|u(t)\|_{L^2}\le\|u\|_{L^\infty_tL^2_x}
\quad\text{for every } t\in[0,T].
\label{eq:weakcontbound}
\end{equation}
Indeed, inserting $\Phi=\eta\phi$ into \eqref{eq:weakform}, with
$\eta\in C_c^\infty([0,T))$ and $\phi$ smooth, divergence-free and compactly supported,
shows that $t\mapsto\int_\Omega u\cdot\phi\,dx$ agrees almost everywhere with an
absolutely continuous function $\ell_\phi$ on $[0,T]$ satisfying
\[
\ell_\phi(0)=\int_\Omega u_0\cdot\phi\,dx,
\qquad
|\ell_\phi(t)|
\le\|u\|_{L^\infty_tL^2_x}\|\phi\|_{L^2}
\quad\text{for every } t\in[0,T].
\]
The uniqueness of continuous representatives makes $\phi\mapsto\ell_\phi(t)$ linear.
By Lemma~\ref{lem:density} and the Riesz representation theorem, these pairings define
$u(t)\in L^2_\sigma(\Omega)$ satisfying \eqref{eq:weakcontbound}.
Using a countable dense family of test functions shows that this is a representative of the
original solution. Finally, density and the uniform bound above imply that
$t\mapsto(u(t),v)_{L^2}$ is continuous for every $v\in L^2_\sigma(\Omega)$.
We always use this representative below.
\end{remark}
\begin{theorem}\label{thm:main}
Let $T>0$, $\nu>0$, and $u_0\in L^2_\sigma(\Omega)$. Let $u$ be a weak solution of
\eqref{eq:NS} on $[0,T]$ with initial datum $u_0$ in the sense of
Definition~\ref{def:weak}, and assume
\begin{equation}
u\in L^4(0,T;L^4(\Omega)),
\label{eq:hypL4}
\end{equation}
\begin{equation}
u\in L^3(0,T;B^{1/3,v}_{3,\infty}(\Omega)).
\label{eq:besovweak}
\end{equation}
Then
\begin{enumerate}
\item[(i)] $\partial_1u\in L^2(0,T;L^2(\Omega))$, and
\begin{equation}
\begin{split}
\nu\int_0^T\|\partial_1u(t)\|_{L^2}^2\,dt
&\;\le\;\|u_0\|_{L^2}^2
+\frac{C}{\nu}\,\|u\|_{L^4(0,T;L^4)}^4\\
&\quad+C\int_0^T\|u(t)\|_{B^{1/3,v}_{3,\infty}}^3\,dt,
\end{split}
\label{eq:apriori}
\end{equation}
with $C$ a universal constant, independent in particular of $\nu$, of $T$ and of $u$.
\end{enumerate}
If, in addition,
\begin{equation}
u\in L^3(0,T;B^{1/3,v}_{3,c_0}(\Omega)),
\label{eq:besovhyp}
\end{equation}
then
\begin{enumerate}
\item[(ii)] $u$ belongs to the isotropic critical Onsager space,
\begin{equation}
u\in L^3\bigl(0,T;B^{1/3}_{3,c_0}(\Omega)\bigr).
\label{eq:isoconc}
\end{equation}
\item[(iii)] The energy equality
\begin{equation}
\frac12\|u(t)\|_{L^2}^2 + \nu\int_0^t \|\partial_1 u(s)\|_{L^2}^2\,ds
= \frac12\|u_0\|_{L^2}^2
\label{eq:energyeq}
\end{equation}
holds for every $t\in[0,T]$, and $u\in C([0,T];L^2(\Omega))$.
\end{enumerate}
\end{theorem}

\begin{remark}\label{rem:shinbrot}
By interpolation with $u\in L^\infty(0,T;L^2(\Omega))$, hypothesis
\eqref{eq:hypL4} is implied by $u\in L^p(0,T;L^q(\Omega))$ with
$2/p+2/q\le1$ and $q\ge4$, so the conditions of Shinbrot type are covered as well.
In particular, the global deterministic solutions constructed in \cite{LZZ} for initial
data $u_0\in\widetilde H^{0,1}(\Omega)$ satisfy \eqref{eq:hypL4} and \eqref{eq:besovhyp}
on every finite time interval.
\end{remark}

\begin{remark}\label{rem:R2}
The following connection with the result of Demmel and Wiedemann
\cite{DemmelWiedemann2026} applies only when $\Omega=\R^2$.
Conclusion (i) of Theorem~\ref{thm:main} shows that $u$ belongs to their energy class, with
$\partial_1u\in L^2(0,T;L^2(\R^2))$. Their energy-rigidity theorem, applied after the
scaling $u(\nu t,\nu x)$ to unit viscosity, then yields the energy equality
\eqref{eq:energyeq} under \eqref{eq:hypL4} and \eqref{eq:besovweak}, without the $c_0$
refinement \eqref{eq:besovhyp}. The proof below treats both domains independently of
their result.
\end{remark}

\section{One-dimensional truncations and commutator estimates}\label{sec:commutator}

We use the horizontal and vertical Littlewood--Paley decompositions of
Definition~\ref{def:blocks}.

The two truncations act in different variables and commute with each other and with all
partial derivatives. Since $0\le\chi\le1$ pointwise, the operators $S_N^h$
and $S_Q^v$ have $L^2(\Omega)\to L^2(\Omega)$ operator norm at most $1$.

The blocks $\Delta_q^h$ and $\Delta_q^v$ localize the Fourier support in one variable
alone. Bernstein's inequality \cite[Lemma~2.1]{BCD} is therefore applied to them in that
variable, for almost every value of the other one, and integration in the latter turns it
into the bound on $\Omega$ used below. On $\T$ the blocks are Fourier series multipliers,
and the periodic Bernstein inequality yields the same estimates.

Each truncation is a convolution in its own variable, against a dilate of one fixed function
on the line, periodized when the factor is compact. Since $\Omega_1=\Omega_2$ in both cases
of \eqref{eq:domains}, the same family of kernels can be used in both directions.

\begin{definition}\label{def:kernels}
Let $\check\chi$ be the inverse Fourier transform of $\chi$, a fixed Schwartz function on
$\R$, and for $M\ge-1$ let
\begin{equation}
\gamma_M(z) := 2^{M+1}\,\check\chi\bigl(2^{M+1}z\bigr),
\qquad z\in\R,
\label{eq:linekernel}
\end{equation}
be the kernel of $\chi(2^{-M-1}D)$ on the real line. Set
\begin{equation}
K_M := \gamma_M \ \text{ if } \Omega_2=\R,
\qquad
K_M := \sum_{k\in\Z}\gamma_M(\cdot+2\pi k) \ \text{ if } \Omega_2=\T,
\label{eq:kernels}
\end{equation}
the periodization being that given by Poisson summation, the series converging
absolutely and uniformly because $\check\chi$ is Schwartz.
\end{definition}

These kernels represent the truncations: for $f\in L^p(\Omega)$ with
$1\le p\le\infty$,
\begin{equation}
S_Q^v f(x) = \int_{\Omega_2} K_Q(z)\, f(x_1,x_2-z)\,dz,
\quad
S_N^h f(x) = \int_{\Omega_1} K_N(z)\, f(x_1-z,x_2)\,dz .
\label{eq:convform}
\end{equation}
When $\Omega_2=\T$, we identify $\T$ with $(-\pi,\pi]$, and $|z|$ denotes the absolute
value of the representative of $z\in\T$ in that interval. In particular
$|z|\le|z+2\pi k|$ for every $z\in(-\pi,\pi]$ and every $k\in\Z$.

\begin{definition}\label{def:comm}
For a function $f$ on $\Omega$ and $z\in\Omega_2$, define its vertical increment by
\begin{equation}
\delta_z^v f(x_1,x_2) := f(x_1,x_2-z) - f(x_1,x_2).
\label{eq:increment}
\end{equation}
For $2\le p<\infty$ and $f\in L^p(\Omega;\R^2)$, define
\begin{equation}
\begin{split}
R_Q^v(f) &:= S_Q^v(f\otimes f) - S_Q^v f\otimes S_Q^v f,\\
R_N^h(f) &:= S_N^h(f\otimes f) - S_N^h f\otimes S_N^h f.
\end{split}
\label{eq:RQv}
\end{equation}
\end{definition}

The assertions of the first lemma are one-dimensional statements about a single variable,
and hold without change on $\T$ and on $\R$, hence for both truncations on both domains.

\begin{lemma}\label{lem:kernelbasic}
For every $M\ge-1$,
\begin{equation}
\|K_M\|_{L^1(\Omega_2)}\lesssim1,
\label{eq:kernelL1}
\end{equation}
\begin{equation}
\int_{\Omega_2}K_M(z)\,dz = 1 .
\label{eq:norm}
\end{equation}
Consequently $S_Q^v$ and $S_N^h$ are uniformly bounded on $L^p(\Omega)$,
$1\le p\le\infty$, and
\begin{equation}
\|S_Q^vf-f\|_{L^p}\to0 \quad(Q\to\infty),
\qquad
\|S_N^hf-f\|_{L^p}\to0 \quad(N\to\infty),
\label{eq:approxid}
\end{equation}
for every $f\in L^p(\Omega)$ with $1\le p<\infty$. Moreover, for every
$f\in\mathscr S'(\Omega)$,
\begin{equation}
S_Q^vf\to f \quad(Q\to\infty),
\qquad
S_N^hf\to f \quad(N\to\infty),
\label{eq:approxSprime}
\end{equation}
with convergence in that space.
\end{lemma}

\begin{proof}
By scaling and periodization, on either domain,
\[
\|K_M\|_{L^1(\Omega_2)}\le\|\gamma_M\|_{L^1(\R)}
=\|\check\chi\|_{L^1(\R)},
\qquad
\int_{\Omega_2}K_M(z)\,dz=\chi(0)=1,
\]
where the last equality follows from \eqref{eq:partition} at $\xi=0$.
The uniform $L^p$ bounds follow by Young's inequality. Also,
\begin{equation}
S_Q^vf-f=\int_{\Omega_2}K_Q(z)\,\delta_z^vf\,dz .
\label{eq:approxdiff}
\end{equation}
For every $\eta>0$,
\[
\int_{\{z\in\Omega_2:\,|z|>\eta\}}|K_Q(z)|\,dz
\le\int_{|y|>\eta}|\gamma_Q(y)|\,dy\longrightarrow0,
\]
where the inequality is an identity on $\R$ and follows on $\T$ by periodization and
$|z+2\pi k|\ge|z|$. Thus \eqref{eq:approxdiff}, Minkowski's inequality and continuity of translation in
$L^p$ for $p<\infty$ yield \eqref{eq:approxid}, and the horizontal case is identical.
Finally, the smooth cutoff multipliers converge to the identity on the test space
$\mathscr S(\Omega)$ by rapid Fourier decay. Since these multipliers are symmetric,
\eqref{eq:approxSprime} follows by duality.
\end{proof}

\begin{lemma}\label{lem:L2comm}
Let $2\le p<\infty$, and let $S_M$ denote either $S_N^h$ or $S_Q^v$.
Correspondingly, $R_M$ denotes $R_N^h$ or $R_Q^v$, respectively, as defined in
\eqref{eq:RQv}.
Then, for every $f\in L^p(\Omega;\R^2)$,
\begin{equation}
\|R_M(f)\|_{L^{p/2}(\Omega)} \lesssim \|f\|_{L^p(\Omega)}^2
\qquad\text{uniformly in } M,
\label{eq:L2commbound}
\end{equation}
and
\begin{equation}
R_M(f)\to0 \quad\text{in } L^{p/2}(\Omega)\text{ as } M\to\infty.
\label{eq:L2commlimit}
\end{equation}
\end{lemma}

\begin{proof}
The estimate follows from the uniform $L^{p/2}$- and $L^p$-boundedness of $S_M$
(Lemma~\ref{lem:kernelbasic}) and $\|f\otimes f\|_{L^{p/2}}=\|f\|_{L^p}^2$. Moreover, by
\eqref{eq:approxid}, $S_M(f\otimes f)\to f\otimes f$ in $L^{p/2}$ and $S_Mf\to f$ in
$L^p$. H\"older's inequality then implies $S_Mf\otimes S_Mf\to f\otimes f$ in $L^{p/2}$, and
hence $R_M(f)\to0$ in $L^{p/2}$.
\end{proof}

\begin{lemma}\label{lem:embedL3}
The vertical Besov space of Definition~\ref{def:besov} embeds continuously into
$L^3(\Omega)$,
\begin{equation}
B^{1/3,v}_{3,\infty}(\Omega)\subset L^3(\Omega),
\qquad
\|f\|_{L^3}\lesssim\|f\|_{B^{1/3,v}_{3,\infty}} .
\label{eq:embedL3}
\end{equation}
\end{lemma}

\begin{proof}
By \eqref{eq:besov} the series $\sum_{j\ge-1}\|\Delta_j^vf\|_{L^3}$ is bounded by
$\|f\|_{B^{1/3,v}_{3,\infty}}\sum_{j\ge-1}2^{-j/3}$, so the partial sums $S_Q^vf$ converge
in $L^3(\Omega)$, and by \eqref{eq:approxSprime} their limit is $f$.
\end{proof}

For $f\in B^{1/3,v}_{3,\infty}(\Omega)$ and $j\ge-1$, write
\[
\epsilon_j(f) := 2^{j/3}\|\Delta_j^v f\|_{L^3},
\]
so that $\sup_{j\ge-1}\epsilon_j(f)=\|f\|_{B^{1/3,v}_{3,\infty}}$ by \eqref{eq:besov}, and
$f\in B^{1/3,v}_{3,c_0}(\Omega)$ precisely when $\epsilon_j(f)\to0$ by \eqref{eq:czero}.
The vertical commutator estimate is stated in terms of the following quantities.

\begin{definition}\label{def:dQ}
For $f\in B^{1/3,v}_{3,\infty}(\Omega)$, which lies in $L^3(\Omega)$ by
Lemma~\ref{lem:embedL3}, and $Q\ge-1$ set
\begin{equation}
\tau_Q(f) := 2^{Q/3}\|f - S_Q^v f\|_{L^3},
\label{eq:tauQ}
\end{equation}
\begin{equation}
T_Q(f) := \left(\int_{\Omega_2} |K_Q(z)|
\bigl(2^{Q/3}\|\delta_z^v f\|_{L^3}\bigr)^2\,dz\right)^{1/2},
\label{eq:TQ}
\end{equation}
\begin{equation}
d_Q(f) := \tau_Q(f) + T_Q(f).
\label{eq:dQ}
\end{equation}
\end{definition}

\begin{lemma}\label{lem:kernelmoment}
For every real $p\ge0$ set
\begin{equation}
\kappa_p:=2^{-p}\int_\R|\check\chi(y)|\,|y|^p\,dy .
\label{eq:kappap}
\end{equation}
Then $\kappa_p<\infty$, and for every $Q\ge-1$,
\begin{equation}
\int_{\Omega_2}|K_Q(z)|\,|z|^p\,dz \;\le\; \kappa_p\,2^{-pQ}
\label{eq:kernelmoment}
\end{equation}
on $\Omega_2=\T$ and on $\Omega_2=\R$.
\end{lemma}

\begin{proof}
On $\Omega_2=\R$, the
change of variables $y=2^{Q+1}z$ shows that
\begin{equation*}
\int_\R|\gamma_Q(z)|\,|z|^p\,dz=\kappa_p\,2^{-pQ}.
\end{equation*}
On $\Omega_2=\T$, the periodization of $\gamma_Q$ and a similar change of variables imply
\begin{equation*}
\begin{aligned}
\int_{-\pi}^{\pi}|K_Q(z)|\,|z|^p\,dz
&\le \sum_{k\in\Z}\int_{-\pi}^{\pi}|\gamma_Q(z+2\pi k)|\,|z|^p\,dz\\
&=\sum_{k\in\Z}\int_{2\pi k-\pi}^{2\pi k+\pi}
|\gamma_Q(y)|\,|y-2\pi k|^p\,dy\\
&\le\int_\R|\gamma_Q(y)|\,|y|^p\,dy
=\kappa_p\,2^{-pQ},
\end{aligned}
\end{equation*}
where the last inequality uses $|y-2\pi k|\le|y|$ on
$[2\pi k-\pi,2\pi k+\pi]$.
\end{proof}

\begin{lemma}\label{lem:increment}
Let $g\in B^{1/3,v}_{3,\infty}(\Omega)$ and let $J\ge-1$ be an integer. Then:
\begin{enumerate}
\item[\textup{(a)}] for every $z\in\Omega_2$,
\begin{equation}
\|\delta_z^v S_J^vg\|_{L^3(\Omega)} \;\lesssim\; 2^{2J/3}\,
\|g\|_{B^{1/3,v}_{3,\infty}}\,|z| .
\label{eq:incrementlow}
\end{equation}
\item[\textup{(b)}] for every $z\in\Omega_2\setminus\{0\}$,
\begin{equation}
\|\delta_z^v (g-S_J^vg)\|_{L^3(\Omega)} \;\lesssim\;
\Bigl(\sup_{j>J}\epsilon_j(g)\Bigr)\,|z|^{1/3}.
\label{eq:incrementtail}
\end{equation}
\end{enumerate}
\end{lemma}

\begin{proof}
The vertical Bernstein inequality reads
\[
\|\partial_2\Delta_j^vg\|_{L^3}
\lesssim 2^j\|\Delta_j^vg\|_{L^3},
\qquad j\ge-1.
\]
Together with translation invariance and the fundamental theorem of calculus, it yields
\begin{equation}
\|\delta_z^v\Delta_j^vg\|_{L^3}
\lesssim\min\{1,2^j|z|\}\,\epsilon_j(g)\,2^{-j/3},
\qquad j\ge-1.
\label{eq:incrementcrude}
\end{equation}

Next, we show that, for every $r>0$,
\begin{equation}
\sum_{j\ge-1}\min\{1,2^jr\}\,2^{-j/3}\lesssim r^{1/3}.
\label{eq:dyadicsum}
\end{equation}
For $0<r<1$, choose $j_r\ge0$ such that $2^{j_r}r\le1<2^{j_r+1}r$. Then
\begin{align*}
\sum_{j\ge-1}\min\{1,2^jr\}\,2^{-j/3}
&=r\sum_{j=-1}^{j_r}2^{2j/3}+\sum_{j>j_r}2^{-j/3}\\
&\lesssim r2^{2j_r/3}+2^{-j_r/3}
\lesssim r^{1/3}.
\intertext{For $r\ge1$, we have}
\sum_{j\ge-1}\min\{1,2^jr\}\,2^{-j/3}
&\le\sum_{j\ge-1}2^{-j/3}\lesssim1\lesssim r^{1/3}.
\end{align*}

For \textup{(a)}, sum \eqref{eq:incrementcrude} over $-1\le j\le J$, using
$\min\{1,2^j|z|\}\le2^j|z|$, $\epsilon_j(g)\le\|g\|_{B^{1/3,v}_{3,\infty}}$ and
$\sum_{j=-1}^J2^{2j/3}\lesssim2^{2J/3}$.

For \textup{(b)}, since $g\in L^3(\Omega)$ by \eqref{eq:embedL3}, the partial sums
$S_Q^vg$ converge to $g$ in $L^3(\Omega)$ by \eqref{eq:approxid}, so that
$g-S_J^vg=\sum_{j>J}\Delta_j^vg$ in $L^3(\Omega)$. Sum \eqref{eq:incrementcrude} over
$j>J$, factor out $\sup_{j>J}\epsilon_j(g)$ and apply \eqref{eq:dyadicsum} with $r=|z|$.
\end{proof}

\begin{lemma}\label{lem:dQprops}
Let $f\in B^{1/3,v}_{3,\infty}(\Omega)$. Then:
\begin{enumerate}
\item[\textup{(a)}] for every $Q\ge-1$,
\begin{equation}
d_Q(f)\;\lesssim\;2^{-Q/3}\|f\|_{B^{1/3,v}_{3,\infty}}+\sup_{j>Q/2}\epsilon_j(f),
\label{eq:dQquant}
\end{equation}
and in particular
\begin{equation}
\sup_{Q\ge-1}d_Q(f) \;\lesssim\; \|f\|_{B^{1/3,v}_{3,\infty}}.
\label{eq:dQbound}
\end{equation}
\item[\textup{(b)}] If $f\in B^{1/3,v}_{3,c_0}(\Omega)$, then
\begin{equation}
\lim_{Q\to\infty}d_Q(f)=0 .
\label{eq:dQiff}
\end{equation}
\end{enumerate}
\end{lemma}

\begin{proof}
Fix $Q\ge-1$. The identity $f-S_Q^vf=\sum_{j>Q}\Delta_j^vf$ and a geometric series
estimate imply
\begin{equation}
\tau_Q(f)
\;\le\;2^{Q/3}\sum_{j>Q}\epsilon_j(f)\,2^{-j/3}
\;\lesssim\;\sup_{j>Q}\epsilon_j(f).
\label{eq:tailL3}
\end{equation}
For $T_Q$, let $J$ be the largest integer with $2J\le Q$, so that $J\ge-1$,
$2^{4J/3}\le2^{2Q/3}$, and $j>J$ exactly when $j>Q/2$. Split $f=S_J^vf+(f-S_J^vf)$. Since $\delta_z^v$ is linear,
the triangle inequalities in $L^3(\Omega)$ and in $L^2(\Omega_2,|K_Q(z)|\,dz)$ show that
$T_Q$ is subadditive, so
\begin{equation}
T_Q(f)\le T_Q(S_J^vf)+T_Q(f-S_J^vf).
\label{eq:TQsubadd}
\end{equation}
Inserting \eqref{eq:incrementlow} into the definition \eqref{eq:TQ} and applying
Lemma~\ref{lem:kernelmoment} at $p=2$,
\begin{equation}
\begin{split}
T_Q(S_J^vf)^2
&\lesssim 2^{2Q/3}\,2^{4J/3}\|f\|_{B^{1/3,v}_{3,\infty}}^2
\int_{\Omega_2}|K_Q(z)|\,|z|^2\,dz\\
&\le \kappa_2\,2^{4J/3}\|f\|_{B^{1/3,v}_{3,\infty}}^2\,2^{-4Q/3}
\;\le\;\kappa_2\,\|f\|_{B^{1/3,v}_{3,\infty}}^2\,2^{-2Q/3},
\end{split}
\label{eq:TQlow}
\end{equation}
while \eqref{eq:incrementtail} and Lemma~\ref{lem:kernelmoment} at $p=2/3$ yield
\begin{equation}
\begin{split}
T_Q(f-S_J^vf)^2
&\lesssim 2^{2Q/3}\Bigl(\sup_{j>J}\epsilon_j(f)\Bigr)^2
\int_{\Omega_2}|K_Q(z)|\,|z|^{2/3}\,dz\\
&\le \kappa_{2/3}\Bigl(\sup_{j>J}\epsilon_j(f)\Bigr)^2 .
\end{split}
\label{eq:TQtail}
\end{equation}
Since $\kappa_2$ and $\kappa_{2/3}$ are universal constants, taking square roots in
\eqref{eq:TQlow} and \eqref{eq:TQtail} and inserting them into \eqref{eq:TQsubadd}, we
obtain
\begin{equation}
T_Q(f)\;\lesssim\;2^{-Q/3}\|f\|_{B^{1/3,v}_{3,\infty}}
\;+\;\sup_{j>Q/2}\epsilon_j(f).
\label{eq:TQquant}
\end{equation}
Combining \eqref{eq:tailL3}, \eqref{eq:TQquant} and \eqref{eq:dQ} yields
\eqref{eq:dQquant}. Since $2^{-Q/3}\le2^{1/3}$ for $Q\ge-1$ and
$\epsilon_j(f)\le\|f\|_{B^{1/3,v}_{3,\infty}}$ for every $j$, \eqref{eq:dQquant} implies
\eqref{eq:dQbound}. If $f\in B^{1/3,v}_{3,c_0}(\Omega)$, then $\epsilon_j(f)\to0$ as
$j\to\infty$ by \eqref{eq:czero}, so both terms on the right of \eqref{eq:dQquant} tend
to $0$ as $Q\to\infty$, which proves \eqref{eq:dQiff}.
\end{proof}

\begin{lemma}\label{lem:comm}
Let $f\in B^{1/3,v}_{3,\infty}(\Omega)$. Then, for every integer $Q\ge-1$,
\begin{equation}
\|R_Q^v(f)\|_{L^{3/2}(\Omega)} \;\le\; 2^{-2Q/3}\, d_Q(f)^2,
\label{eq:RQthreehalf}
\end{equation}
\begin{equation}
\|\partial_2 S_Q^v f\|_{L^3(\Omega)} \;\lesssim\; 2^{2Q/3}\,\|f\|_{B^{1/3,v}_{3,\infty}}.
\label{eq:vertderiv}
\end{equation}
Consequently,
\begin{equation}
\|R_Q^v(f)\|_{L^{3/2}(\Omega)}\,\|\partial_2 S_Q^v f\|_{L^3(\Omega)}
\;\lesssim\; d_Q(f)^2\,\|f\|_{B^{1/3,v}_{3,\infty}}
\label{eq:commproduct}
\end{equation}
uniformly in $Q$.
\end{lemma}

\begin{proof}
By \eqref{eq:embedL3}, $f\in L^3(\Omega)$, so that $f\otimes f\in L^{3/2}(\Omega)$.
Using the convolution formula \eqref{eq:convform} and \eqref{eq:norm},
we can rewrite \eqref{eq:RQv} in $L^{3/2}(\Omega)$ as
\begin{equation}
R_Q^v(f)
= \int_{\Omega_2}K_Q(z)\,\delta_z^vf\otimes\delta_z^vf\,dz
-(f-S_Q^vf)\otimes(f-S_Q^vf).
\label{eq:commid}
\end{equation}
Applying to \eqref{eq:commid} the triangle inequality, Minkowski's integral
inequality in $L^{3/2}(\Omega)$, and the identity
$\|g\otimes g\|_{L^{3/2}}=\|g\|_{L^3}^2$, we obtain
\[
\|R_Q^v(f)\|_{L^{3/2}}
\le \int_{\Omega_2}|K_Q(z)|\,\|\delta_z^vf\|_{L^3}^2\,dz
+\|f-S_Q^vf\|_{L^3}^2.
\]
Multiplying by $2^{2Q/3}$ turns the two terms on the right into $T_Q(f)^2$ and
$\tau_Q(f)^2$ respectively, by \eqref{eq:tauQ} and \eqref{eq:TQ}, so that
$s^2+t^2\le(s+t)^2$ for $s,t\ge0$ and \eqref{eq:dQ} imply
\[
2^{2Q/3}\|R_Q^v(f)\|_{L^{3/2}}
\le T_Q(f)^2+\tau_Q(f)^2
\le d_Q(f)^2,
\]
which proves \eqref{eq:RQthreehalf}.

To prove \eqref{eq:vertderiv}, write $S_Q^vf=\sum_{-1\le q\le Q}\Delta_q^vf$.
By Bernstein's inequality \cite[Lemma~2.1]{BCD} and the definition of $\epsilon_q(f)$,
\[
\|\partial_2\Delta_q^vf\|_{L^3(\Omega)}
\;\lesssim\; 2^{q}\,\|\Delta_q^vf\|_{L^3(\Omega)}
\;=\; 2^{2q/3}\,\epsilon_q(f),
\qquad -1\le q\le Q .
\]
Summing over $-1\le q\le Q$ and using \eqref{eq:besov}, we obtain
\[
\|\partial_2 S_Q^v f\|_{L^3}
\lesssim \|f\|_{B^{1/3,v}_{3,\infty}}\sum_{q=-1}^{Q}2^{2q/3}
\lesssim 2^{2Q/3}\|f\|_{B^{1/3,v}_{3,\infty}},
\]
which proves \eqref{eq:vertderiv}.
Multiplying \eqref{eq:RQthreehalf} by \eqref{eq:vertderiv} gives \eqref{eq:commproduct}.
\end{proof}

\section{Horizontal dissipation, isotropic regularity, and the energy equality}\label{sec:apriori-energy}

\subsection{The horizontal dissipation estimate}\label{subsec:apriori-estimate}

Throughout this subsection, $u$ is a weak solution in the sense of
Definition~\ref{def:weak} satisfying \eqref{eq:hypL4} and \eqref{eq:besovweak}. The
$c_0$ refinement \eqref{eq:besovhyp} is used only in
Subsection~\ref{subsec:isotropic-energy}. Set
\[
u_Q := S_Q^v u,\qquad u_{N,Q} := S_N^h S_Q^v u = S_Q^v S_N^h u.
\]
Define $R_{N,Q}$ and $\Pi_{N,Q}$ by
\begin{equation}
\begin{split}
R_{N,Q}(u) &:= S_N^h S_Q^v (u\otimes u) - u_{N,Q}\otimes u_{N,Q},\\
\Pi_{N,Q} &:= \int_\Omega R_{N,Q}(u):\nabla u_{N,Q}\,dx .
\end{split}
\label{eq:PiNQ}
\end{equation}
By Lemma~\ref{lem:balance} in Appendix~\ref{app:A}, we obtain the \emph{doubly truncated
energy balance} for all
$N,Q\ge-1$ and every $t\in[0,T]$:
\begin{equation}
\frac12\|u_{N,Q}(t)\|_{L^2}^2 + \nu\int_0^t \|\partial_1 u_{N,Q}(s)\|_{L^2}^2\,ds
= \frac12\|u_{N,Q}(0)\|_{L^2}^2 + \int_0^t \Pi_{N,Q}(s)\,ds.
\label{eq:balanceNQ}
\end{equation}
The commutator can be split into a horizontal and a vertical part.

\begin{lemma}\label{lem:splitcomm}
Let $f\in L^4(\Omega;\R^2)$ and let $N,Q\ge-1$. Then
\begin{equation}
S_N^hS_Q^v(f\otimes f)-S_N^hS_Q^vf\otimes S_N^hS_Q^vf
= S_N^h R_Q^v(f) + R_N^h(S_Q^vf),
\label{eq:splitcomm}
\end{equation}
and, uniformly in $N$ and $Q$,
\begin{equation}
\bigl\|S_N^hS_Q^v(f\otimes f)-S_N^hS_Q^vf\otimes S_N^hS_Q^vf\bigr\|_{L^2}
\;\lesssim\; \|f\|_{L^4}^2 .
\label{eq:RNQL2}
\end{equation}
\end{lemma}

\begin{proof}
Expanding $S_N^hR_Q^v(f)$ and $R_N^h(S_Q^vf)$ by \eqref{eq:RQv}, the two occurrences of
$S_N^h(S_Q^vf\otimes S_Q^vf)$ cancel, which is \eqref{eq:splitcomm}. For
\eqref{eq:RNQL2}, the triangle inequality, the
$L^2$-contractivity of $S_N^h$ and $S_Q^v$, and their uniform $L^4$-boundedness
(Lemma~\ref{lem:kernelbasic}) yield
\[
\|S_N^hS_Q^v(f\otimes f)\|_{L^2}+\|S_N^hS_Q^vf\otimes S_N^hS_Q^vf\|_{L^2}
\le \|f\otimes f\|_{L^2}+\|S_N^hS_Q^vf\|_{L^4}^2
\lesssim \|f\|_{L^4}^2 .
\]
\end{proof}

The following proposition contains the main a priori estimate of the paper and proves part (i) of Theorem~\ref{thm:main}.

\begin{proposition}\label{prop:apriori}
Let $u$ be a weak solution in the sense of Definition~\ref{def:weak} satisfying
\eqref{eq:hypL4} and \eqref{eq:besovweak}. Then
$\partial_1u\in L^2(0,T;L^2(\Omega))$ and \eqref{eq:apriori} holds.
\end{proposition}

\begin{proof}
\emph{Step 1: the horizontal flux.}
For a matrix $R$, write $R^{\cdot k}:=Re_k$ for its $k$-th column, $(e_1,e_2)$ being the
standard basis of $\R^2$.
Split $\Pi_{N,Q}=\Pi^h_{N,Q}+\Pi^v_{N,Q}$ with
\[
\Pi^h_{N,Q} = \int_\Omega (R_{N,Q}(u))^{\cdot 1}\cdot\partial_1u_{N,Q}\,dx,
\qquad
\Pi^v_{N,Q} = \int_\Omega (R_{N,Q}(u))^{\cdot 2}\cdot\partial_2u_{N,Q}\,dx.
\]
By \eqref{eq:RNQL2}, Cauchy--Schwarz and Young's inequality, for a.e.\ $t$,
\begin{equation}
\begin{split}
|\Pi^h_{N,Q}(t)|
&\le \|R_{N,Q}(u(t))\|_{L^2}\,\|\partial_1u_{N,Q}(t)\|_{L^2}\\
&\le \frac{\nu}{2}\,\|\partial_1u_{N,Q}(t)\|_{L^2}^2
+ \frac{C}{\nu}\,\|u(t)\|_{L^4}^4 .
\end{split}
\label{eq:young}
\end{equation}
By Bernstein's inequality and \eqref{eq:weakcontbound},
$\|\partial_1u_{N,Q}(t)\|_{L^2}\lesssim2^N\|u\|_{L^\infty_tL^2_x}$ for every $t$, so that
$\int_0^T\|\partial_1u_{N,Q}\|_{L^2}^2\,dt$ is finite for each fixed pair $(N,Q)$. We may
therefore apply \eqref{eq:young} in \eqref{eq:balanceNQ} at $t=T$ and move the resulting
dissipation term to the left. Dropping the nonnegative term
$\tfrac12\|u_{N,Q}(T)\|_{L^2}^2$ and using $\|u_{N,Q}(0)\|_{L^2}\le\|u_0\|_{L^2}$,
we obtain
\begin{equation}
\begin{split}
\frac{\nu}{2}\int_0^T\|\partial_1u_{N,Q}\|_{L^2}^2\,dt
&\;\le\;\frac12\|u_0\|_{L^2}^2
+\frac{C}{\nu}\int_0^T\|u\|_{L^4}^4\,dt\\
&\quad+\int_0^T|\Pi^v_{N,Q}(t)|\,dt .
\end{split}
\label{eq:absorbed}
\end{equation}

\emph{Step 2: the vertical flux.}
By Lemma~\ref{lem:splitcomm},
\[
\Pi^v_{N,Q} = \int_\Omega (S_N^hR_Q^v(u))^{\cdot 2}\cdot\partial_2u_{N,Q}\,dx
+ \int_\Omega (R_N^h(u_Q))^{\cdot 2}\cdot\partial_2u_{N,Q}\,dx
=: I_{N,Q} + J_{N,Q}.
\]
For $I_{N,Q}$, H\"older's inequality, the uniform $L^{3/2}$- and
$L^3$-boundedness of $S_N^h$ (Lemma~\ref{lem:kernelbasic}), \eqref{eq:commproduct} and
\eqref{eq:dQbound} yield, uniformly in $N$ and $Q$,
\begin{equation}
|I_{N,Q}(t)|
\le \|S_N^hR_Q^v(u(t))\|_{L^{3/2}}\,\|S_N^h\partial_2u_Q(t)\|_{L^3}
\lesssim d_Q(u(t))^2\,\|u(t)\|_{B^{1/3,v}_{3,\infty}}
\lesssim \|u(t)\|_{B^{1/3,v}_{3,\infty}}^3,
\label{eq:Ibound}
\end{equation}
and the right-hand side lies in $L^1(0,T)$ by \eqref{eq:besovweak}.

For $J_{N,Q}$, fix $Q$. For a.e.\ $t$, we have $u_Q(t)\in L^4(\Omega)$, so
Lemma~\ref{lem:L2comm} with $p=4$ gives both
\[
\|R_N^h(u_Q(t))\|_{L^2}\lesssim\|u_Q(t)\|_{L^4}^2
\lesssim\|u(t)\|_{L^4}^2
\qquad\text{uniformly in }N,
\]
and $R_N^h(u_Q(t))\to0$ in $L^2(\Omega)$ as $N\to\infty$.
The second inequality uses the uniform $L^4$-boundedness of $S_Q^v$
(Lemma~\ref{lem:kernelbasic}). Moreover, since $\partial_2u_{N,Q}=S_N^h\partial_2u_Q$, the
$L^2$-contractivity of $S_N^h$ and Bernstein's inequality \cite[Lemma~2.1]{BCD} give
\[
\|\partial_2u_{N,Q}(t)\|_{L^2}\le\|\partial_2u_Q(t)\|_{L^2}\lesssim2^Q\|u(t)\|_{L^2}
\]
uniformly in $N$. Hence $J_{N,Q}(t)\to0$ as $N\to\infty$ for a.e.\ $t$, with the bound
\begin{equation*}
\begin{aligned}
|J_{N,Q}(t)|&\le \|R_N^h(u_Q(t))\|_{L^2}\,\|\partial_2u_{N,Q}(t)\|_{L^2}\\
&\lesssim 2^Q\,\|u\|_{L^\infty_tL^2_x}\,\|u(t)\|_{L^4}^2 \in L^1(0,T).
\end{aligned}
\end{equation*}
By dominated convergence,
\begin{equation}
\int_0^T |J_{N,Q}(t)|\,dt \;\longrightarrow\; 0
\qquad (N\to\infty,\ Q \text{ fixed}).
\label{eq:Jzero}
\end{equation}

\emph{Step 3: passage to the limit.}
Combining \eqref{eq:absorbed}, \eqref{eq:Ibound} and \eqref{eq:Jzero},
\begin{equation*}
\begin{split}
\limsup_{N\to\infty}\ \int_0^T\|\partial_1u_{N,Q}\|_{L^2}^2\,dt
&\;\le\;\frac{2}{\nu}\Bigl(\frac12\|u_0\|_{L^2}^2
+\frac{C}{\nu}\,\|u\|_{L^4(0,T;L^4)}^4\\
&\qquad\qquad+C\int_0^T\|u(t)\|_{B^{1/3,v}_{3,\infty}}^3\,dt\Bigr) \;=:\; K,
\end{split}
\end{equation*}
with $K$ independent of $Q$. By \eqref{eq:approxid} and dominated convergence,
$u_{N,Q}\to u_Q$ as $N\to\infty$ for fixed $Q$, and $u_Q\to u$ as $Q\to\infty$,
both in $L^2((0,T)\times\Omega)$. Their derivatives therefore converge in distributions.
Applying weak compactness and lower semicontinuity in $L^2$ first as $N\to\infty$
and then as $Q\to\infty$, we obtain $\partial_1u\in L^2((0,T)\times\Omega)$ and
\[
\int_0^T\|\partial_1u\|_{L^2}^2\,dt
\le\liminf_{Q\to\infty}\int_0^T\|\partial_1u_Q\|_{L^2}^2\,dt\le K,
\]
which is \eqref{eq:apriori}.
\end{proof}

\subsection{Isotropic regularity and the energy equality}\label{subsec:isotropic-energy}

This subsection proves parts (ii) and (iii) of Theorem~\ref{thm:main}, both of which use
the $c_0$ refinement \eqref{eq:besovhyp}. By Proposition~\ref{prop:apriori} we may use
$\partial_1u\in L^2(0,T;L^2(\Omega))$ throughout.

\begin{proposition}\label{prop:isotropic}
Under the hypotheses of Theorem~\ref{thm:main}, including the additional assumption
\eqref{eq:besovhyp},
\[
u\in L^3\bigl(0,T;B^{1/3}_{3,c_0}(\Omega)\bigr).
\]
\end{proposition}

\begin{proof}
For $q\ge0$, interpolation between $L^2$ and $L^4$, followed by Plancherel's theorem using
$|\xi_1|\sim2^q$ on the support of $\varphi(2^{-q}\xi_1)$, implies
\begin{equation}
2^{q/3}\|\Delta_q^hu\|_{L^3}
\le\bigl(2^{q}\|\Delta_q^hu\|_{L^2}\bigr)^{1/3}\|\Delta_q^hu\|_{L^4}^{2/3}
\lesssim \|\Delta_q^h\partial_1u\|_{L^2}^{1/3}\,\|u\|_{L^4}^{2/3}.
\label{eq:horizprofile}
\end{equation}
Moreover, by Proposition~\ref{prop:apriori}, \eqref{eq:hypL4}, and Cauchy--Schwarz in time,
\begin{equation}
\begin{split}
\int_0^T\left(\sup_{q\ge0}2^{q/3}\|\Delta_q^hu(t)\|_{L^3}\right)^3dt
&\lesssim \int_0^T\|\partial_1u(t)\|_{L^2}\,\|u(t)\|_{L^4}^2\,dt\\
&\le \|\partial_1u\|_{L^2(0,T;L^2)}
\|u\|_{L^4(0,T;L^4)}^2<\infty.
\end{split}
\label{eq:horizprofile-time}
\end{equation}
Since $\varphi(2^{-q}\xi_1)$ vanishes for $|\xi_1|<\tfrac34 2^q$, Plancherel's theorem and
dominated convergence show that $\|\Delta_q^h\partial_1u(t)\|_{L^2}\to0$ as $q\to\infty$
for a.e.\ $t$, so the right-hand side of \eqref{eq:horizprofile} tends to zero.
Set $S_{-2}^h=S_{-2}^v=0$. Summing the differences
$S_m^hS_m^v-S_{m-1}^hS_{m-1}^v$ yields
\begin{equation}
\mathrm{Id}\;=\;\sum_{m\ge-1}\bigl(\Delta_m^hS_m^v+S_{m-1}^h\Delta_m^v\bigr).
\label{eq:idsplit}
\end{equation}
The series converges strongly on $L^2(\Omega)$ by \eqref{eq:approxid}.
By Fourier support localization, we have
\begin{equation*}
\Delta_qu=\Delta_q\!\!\sum_{\substack{m\ge-1\\|m-q|\le5}}
\bigl(\Delta_m^hS_m^vu+S_{m-1}^h\Delta_m^vu\bigr),
\qquad q\ge0,
\end{equation*}
and therefore, by the uniform $L^3$-boundedness of the blocks $\Delta_q$,
\begin{equation}
2^{q/3}\|\Delta_qu\|_{L^3}
\lesssim\!\!\sum_{\substack{m\ge-1\\|m-q|\le5}}2^{m/3}
\bigl(\|\Delta_m^hu\|_{L^3}+\|\Delta_m^vu\|_{L^3}\bigr).
\label{eq:isoprofile}
\end{equation}
By \eqref{eq:embedL3}, the $L^3$-norms of $\Delta_{-1}u$, $\Delta_{-1}^hu$ and
$\Delta_{-1}^vu$ are bounded by $C\|u\|_{B^{1/3,v}_{3,\infty}}$.
The sum in \eqref{eq:isoprofile} has at most $11$ terms. Hence
\[
\sup_{q\ge-1}2^{q/3}\|\Delta_qu(t)\|_{L^3}
\;\lesssim\;\sup_{m\ge0}2^{m/3}\|\Delta_m^hu(t)\|_{L^3}
\;+\;\|u(t)\|_{B^{1/3,v}_{3,\infty}}.
\]
Both terms on the right lie in $L^3(0,T)$ by \eqref{eq:horizprofile-time} and
\eqref{eq:besovweak}.
For a.e.\ $t$, we have $2^{m/3}\|\Delta_m^hu(t)\|_{L^3}\to0$ as $m\to\infty$ by
\eqref{eq:horizprofile}. Similarly, $2^{m/3}\|\Delta_m^vu(t)\|_{L^3}\to0$ by
\eqref{eq:besovhyp}.
Since $|m-q|\le5$, \eqref{eq:isoprofile} implies
$2^{q/3}\|\Delta_qu(t)\|_{L^3}\to0$ as $q\to\infty$. Hence
$u\in L^3(0,T;B^{1/3}_{3,c_0}(\Omega))$.
\end{proof}

We turn to part (iii). By Proposition~\ref{prop:apriori} and the commutation of $S_Q^v$
with $\partial_1$,
\begin{equation}
\partial_1u_Q=S_Q^v\partial_1u\in L^2((0,T)\times\Omega),
\qquad
\partial_1u_{Q}\longrightarrow\partial_1u \quad (Q\to\infty)
\quad\text{in } L^2((0,T)\times\Omega),
\label{eq:d1uQconv}
\end{equation}
where convergence follows from \eqref{eq:approxid} applied to $\partial_1u$ and
dominated convergence. Set
\begin{equation}
\Pi_Q := \int_{\Omega} R_Q^v(u):\nabla u_Q\,dx .
\label{eq:PiQ}
\end{equation}

\begin{lemma}\label{lem:balanceQ}
Let $u$ be a weak solution in the sense of Definition~\ref{def:weak} satisfying
\eqref{eq:hypL4}, and assume $\partial_1u\in L^2(0,T;L^2(\Omega))$. Then, for each fixed
$Q\ge-1$, $\Pi_Q\in L^1(0,T)$ and the vertically truncated energy balance
\begin{equation}
\frac12\|u_Q(t)\|_{L^2}^2 + \nu\int_0^t \|\partial_1 u_Q(s)\|_{L^2}^2\,ds
= \frac12\|u_Q(0)\|_{L^2}^2 + \int_0^t \Pi_Q(s)\,ds
\label{eq:balanceQ}
\end{equation}
holds for every $t\in[0,T]$.
\end{lemma}

\begin{proof}
Fix $Q\ge-1$. By \eqref{eq:d1uQconv}, Bernstein's inequality and
Lemma~\ref{lem:L2comm} with $p=4$, we have
\begin{equation}
\nabla u_Q\in L^2((0,T)\times\Omega),
\qquad
\|R_Q^v(u(t))\|_{L^2}\lesssim\|u(t)\|_{L^4}^2\in L^2(0,T).
\label{eq:balanceQmemb}
\end{equation}
In particular $\Pi_Q\in L^1(0,T)$ by Cauchy--Schwarz.

For every $t\in[0,T]$, \eqref{eq:approxid} gives
$u_{N,Q}(t)=S_N^hu_Q(t)\to u_Q(t)$ in $L^2(\Omega)$ as $N\to\infty$.
Thus the kinetic energy converges, including at $t=0$.
Also, \eqref{eq:approxid} and dominated convergence give
$\partial_1u_{N,Q}=S_N^h\partial_1u_Q\to\partial_1u_Q$ in
$L^2((0,T)\times\Omega)$ as $N\to\infty$.
Hence, for every $t\in[0,T]$,
\[
\nu\int_0^t\|\partial_1u_{N,Q}(s)\|_{L^2}^2\,ds
\longrightarrow
\nu\int_0^t\|\partial_1u_Q(s)\|_{L^2}^2\,ds
\qquad(N\to\infty).
\]

For the flux, Lemma~\ref{lem:splitcomm} splits
\[
\Pi_{N,Q}=\int_\Omega S_N^hR_Q^v(u):\nabla u_{N,Q}\,dx
+\int_\Omega R_N^h(u_Q):\nabla u_{N,Q}\,dx=:\Pi'_{N,Q}+\Pi''_{N,Q}.
\]
By \eqref{eq:balanceQmemb}, \eqref{eq:approxid} and dominated convergence, $S_N^hR_Q^v(u)\to R_Q^v(u)$
and $\nabla u_{N,Q}=S_N^h\nabla u_Q\to\nabla u_Q$ in $L^2((0,T)\times\Omega)$, so that
Cauchy--Schwarz in $(t,x)$ yields
\begin{equation}
\int_0^T|\Pi'_{N,Q}-\Pi_Q|\,dt\longrightarrow0\qquad(N\to\infty).
\label{eq:Piprime}
\end{equation}
For $\Pi''_{N,Q}$, Lemma~\ref{lem:L2comm} with $p=4$ gives
$R_N^h(u_Q(t))\to0$ in $L^2(\Omega)$ as $N\to\infty$ for a.e.\ $t$.
The $L^2$-contractivity of $S_N^h$ gives
$\|\nabla u_{N,Q}(t)\|_{L^2}\le\|\nabla u_Q(t)\|_{L^2}$.
Hence $\Pi''_{N,Q}(t)\to0$ for a.e.\ $t$.
The commutator bound in the same lemma also yields
\[
|\Pi''_{N,Q}(t)|\lesssim\|u(t)\|_{L^4}^2\,\|\nabla u_Q(t)\|_{L^2}.
\]
This bound is uniform in $N$. Its right-hand side belongs to $L^1(0,T)$ by
\eqref{eq:balanceQmemb}. Dominated convergence therefore gives
\begin{equation}
\int_0^T|\Pi''_{N,Q}|\,dt\longrightarrow0\qquad(N\to\infty).
\label{eq:Pidoubleprime}
\end{equation}
By \eqref{eq:Piprime} and \eqref{eq:Pidoubleprime},
$\int_0^t\Pi_{N,Q}\,ds\to\int_0^t\Pi_Q\,ds$ uniformly in $t\in[0,T]$ as $N\to\infty$.
The identity \eqref{eq:balanceNQ} holds by Lemma~\ref{lem:balance}.
Letting $N\to\infty$ in this identity proves \eqref{eq:balanceQ} for every $t\in[0,T]$.
\end{proof}

\begin{proof}[Proof of Theorem~\ref{thm:main}]
Part (i) is Proposition~\ref{prop:apriori}. Under the additional assumption
\eqref{eq:besovhyp}, part (ii) is Proposition~\ref{prop:isotropic}. For part (iii), split
$\Pi_Q = \Pi_Q^h + \Pi_Q^v$ with
\[
\Pi_Q^h = \int_{\Omega} (R_Q^v(u))^{\cdot 1}\cdot\partial_1 u_Q\,dx,
\qquad
\Pi_Q^v = \int_{\Omega} (R_Q^v(u))^{\cdot 2}\cdot\partial_2 u_Q\,dx.
\]

\emph{Step 1: the vertical flux.}
By H\"older's inequality and \eqref{eq:commproduct},
\begin{equation*}
|\Pi_Q^v(t)| \le \|R_Q^v(u(t))\|_{L^{3/2}} \|\partial_2 u_Q(t)\|_{L^3}
\lesssim d_Q(u(t))^2\,\|u(t)\|_{B^{1/3,v}_{3,\infty}}.
\end{equation*}
For a.e.\ $t$, $u(t)\in B^{1/3,v}_{3,c_0}(\Omega)$ by \eqref{eq:besovhyp}, so
$d_Q(u(t))\to0$ by \eqref{eq:dQiff}, while
$d_Q(u(t))^2\|u(t)\|_{B^{1/3,v}_{3,\infty}}\lesssim\|u(t)\|_{B^{1/3,v}_{3,\infty}}^3\in L^1(0,T)$
by \eqref{eq:dQbound} and \eqref{eq:besovweak}. Dominated convergence implies
\begin{equation}
\int_0^T |\Pi_Q^v(t)|\,dt \to 0 \qquad (Q\to\infty).
\label{eq:vfluxzero}
\end{equation}

\emph{Step 2: the horizontal flux.}
By Cauchy--Schwarz,
Lemma~\ref{lem:L2comm}, and the uniform $L^2$-boundedness of $S_Q^v$ applied to
$\partial_1u(t)$,
\begin{equation}
|\Pi_Q^h(t)| \le \|R_Q^v(u(t))\|_{L^2}\,\|\partial_1u_Q(t)\|_{L^2}
\lesssim \|u(t)\|_{L^4}^2 \|\partial_1 u(t)\|_{L^2} =: G(t),
\label{eq:envelope}
\end{equation}
uniformly in $Q$. By Cauchy--Schwarz in $t$,
\[
\int_0^T G(t)\,dt \le \|u\|_{L^4_tL^4_x}^2 \, \|\partial_1 u\|_{L^2_tL^2_x} < \infty
\]
by \eqref{eq:hypL4} and Proposition~\ref{prop:apriori}, so $G\in L^1(0,T)$. Moreover, for
a.e.\ $t$,
Lemma~\ref{lem:L2comm} applied to $S_Q^v$ shows that $R_Q^v(u(t))\to0$ strongly in
$L^2(\Omega)$, so $\Pi_Q^h(t)\to0$ by the first inequality in \eqref{eq:envelope}. Dominated
convergence then yields
\begin{equation}
\int_0^T |\Pi_Q^h(t)|\,dt \to 0 \qquad (Q\to\infty).
\label{eq:hfluxzero}
\end{equation}

\emph{Step 3: conclusion.}
Combining \eqref{eq:vfluxzero} and \eqref{eq:hfluxzero},
\begin{equation}
\sup_{0\le t\le T} \Bigl|\int_0^t \Pi_Q(s)\,ds\Bigr| \le \int_0^T |\Pi_Q(s)|\,ds \to 0.
\label{eq:supzero}
\end{equation}
Since $u(t)\in L^2(\Omega)$ for every $t\in[0,T]$ and $u(0)=u_0$,
\eqref{eq:approxid} yields $u_Q(t)\to u(t)$ and
$u_Q(0)\to u_0$ strongly in $L^2(\Omega)$ for every $t\in[0,T]$.
Passing to the limit $Q\to\infty$ in \eqref{eq:balanceQ}, using \eqref{eq:d1uQconv} and
\eqref{eq:supzero}, we obtain \eqref{eq:energyeq} for every $t\in[0,T]$. Finally, the
dissipation integral in \eqref{eq:energyeq} is continuous in $t$ because
$\partial_1u\in L^2(0,T;L^2(\Omega))$, so \eqref{eq:energyeq} makes
$t\mapsto\|u(t)\|_{L^2}$ continuous on $[0,T]$. Together with weak continuity,
it follows that $u\in C([0,T];L^2(\Omega))$.
\end{proof}

\backmatter

\section*{Declarations}

\begin{itemize}
\item \textbf{Funding.}\\
The author is supported by the Basic Research Program of Jiangsu (No. BK20251436) and the
Fundamental Research Funds for the Central Universities in China (No. 30924010943).
\item \textbf{Competing interests.}\\
The author declares no competing interests.
\item \textbf{Data availability.}\\
Data sharing not applicable to this article as no datasets were generated or analysed
during the current study.
\item \textbf{Use of AI tools.}\\
The author used ChatGPT (OpenAI) and Claude (Anthropic) for language editing, literature search, checking of references, and suggestions on the organization and presentation of the manuscript. The author independently verified all mathematical content and references and takes full responsibility for the paper.
\end{itemize}

\begin{appendices}

\section{The doubly truncated energy balance}\label{app:A}

Write
\[
\mathscr D(\Omega):=
\begin{cases}
C_c^\infty(\Omega), & \Omega=\R^2,\\
C^\infty(\Omega), & \Omega=\Torus,
\end{cases}
\]
and let $\mathscr D_\sigma(\Omega)$ and $H^1_\sigma(\Omega)$ be the divergence-free fields
in $\mathscr D(\Omega;\R^2)$ and in $H^1(\Omega;\R^2)$, and write
$\nabla^\perp:=(-\partial_2,\partial_1)$. On $\Omega=\R^2$ we use the cutoff
\begin{equation}
\Theta_R(x):=\theta(x_1/R)\,\theta(x_2/R),
\qquad \theta\in C_c^\infty(\R),
\quad \operatorname{supp}\theta\subset[-2,2],
\quad \theta\equiv1 \text{ on } [-1,1],
\label{eq:cutoff}
\end{equation}
with $\theta$ fixed once and for all. Universal constants may depend on $\theta$, as
stated in Section~\ref{sec:statement}. We also use the two-dimensional Ladyzhenskaya
inequality \cite{Ladyzhenskaya1969},
\begin{equation}
\|f\|_{L^4(\Omega)}\lesssim\|f\|_{L^2(\Omega)}^{1/2}\,\|f\|_{H^1(\Omega)}^{1/2},
\qquad\text{so in particular}\qquad H^1(\Omega)\hookrightarrow L^4(\Omega),
\label{eq:ladyzh}
\end{equation}
valid on both domains. We recall the following standard density result and include a
proof for completeness.

\begin{lemma}\label{lem:density}
$\mathscr D_\sigma(\Omega)$ is dense in $H^1_\sigma(\Omega)$ and in $L^2_\sigma(\Omega)$.
\end{lemma}

\begin{proof}
On $\Omega=\Torus$ this follows from Fourier truncation. On $\Omega=\R^2$, let $v\in H^1_\sigma(\R^2)$. The field $(v_2,-v_1)\in H^1(\R^2;\R^2)$ is
curl-free because $\nabla\cdot v=0$, hence a distributional gradient $\nabla\psi$ on the
simply connected plane, and $\nabla\psi\in H^1(\R^2;\R^2)$ shows that $\psi\in
H^2_{\mathrm{loc}}(\R^2)$, with $v=\nabla^\perp\psi$. Set
\[
\begin{gathered}
A_R:=(-2R,2R)^2\setminus[-R,R]^2,\qquad
c_R:=\frac{1}{|A_R|}\int_{A_R}\psi\,dx,\\
v_R:=\nabla^\perp\bigl(\Theta_R(\psi-c_R)\bigr),
\end{gathered}
\]
a compactly supported divergence-free field. The scaled Poincar\'e inequality on $A_R$
bounds $\|\psi-c_R\|_{L^2(A_R)}$ by $CR\|v\|_{L^2(A_R)}$, and together with
$\|\nabla^k\Theta_R\|_{L^\infty}\lesssim R^{-k}$ for $k=1,2$ this yields
\[
\|v_R-v\|_{H^1(\R^2)}
\lesssim\Bigl(\|v\|_{H^1(\R^2\setminus[-R,R]^2)}+R^{-1}\|v\|_{L^2(A_R)}\Bigr)
\longrightarrow0 .
\]
Mollifying $v_R$ and taking a diagonal sequence proves density in $H^1_\sigma(\Omega)$.
For $v\in L^2_\sigma(\Omega)$, the approximants $S_N^hS_N^vv$ belong to
$H^1_\sigma(\Omega)$ by Bernstein's inequality and converge to $v$ in $L^2$ by
\eqref{eq:approxid}. Density in $L^2_\sigma(\Omega)$ follows.
\end{proof}

\begin{lemma}\label{lem:balance}
Let $u$ be a weak solution in the sense of Definition~\ref{def:weak} satisfying
\eqref{eq:hypL4}, and let $N,Q\ge-1$. Then $\Pi_{N,Q}\in L^1(0,T)$ and
\eqref{eq:balanceNQ} holds for every $t\in[0,T]$.
\end{lemma}

\begin{proof}
By Bernstein's inequality, $u\in L^\infty_tL^2_x$ and \eqref{eq:hypL4} give
\begin{equation}
u_{N,Q}\in L^\infty(0,T;H^2(\Omega)),
\qquad
S_N^hS_Q^v(u\otimes u)\in L^2(0,T;H^1(\Omega)).
\label{eq:wAreg}
\end{equation}
Let $\eta\in C_c^\infty(0,T)$ and $v\in\mathscr D_\sigma(\Omega)$. On $\Torus$,
$\eta\,S_N^hS_Q^vv$ is an admissible test function in \eqref{eq:weakform}. On $\R^2$,
$S_N^hS_Q^vv$ is a Schwartz function. Since $\operatorname{div}v=0$ and $v$ has
compact support, $\int_\R v_1(x_1,y)\,dy=0$ for every $x_1$, so
\[
\psi(x):=-\int_{-\infty}^{x_2}v_1(x_1,y)\,dy
\]
belongs to $C_c^\infty(\R^2)$ and satisfies $v=\nabla^\perp\psi$. Since
$S_N^hS_Q^v\psi$ is a Schwartz function, the fields
$\nabla^\perp(\Theta_RS_N^hS_Q^v\psi)\in\mathscr D_\sigma(\R^2)$ converge to
$S_N^hS_Q^vv=\nabla^\perp S_N^hS_Q^v\psi$ in $H^2(\R^2)$ as $R\to\infty$. As
$u\in L^\infty_tL^2_x$ and $u\otimes u\in L^2_tL^2_x$, \eqref{eq:weakform} holds for
$\Phi=\eta\,S_N^hS_Q^vv$ on both domains.

Let $\mathbb P$ be the Leray projection. Since $S_N^hS_Q^v$ is symmetric and commutes
with spatial derivatives, taking $\Phi=\eta\,S_N^hS_Q^vv$ in \eqref{eq:weakform},
integrating by parts in $x$ by \eqref{eq:wAreg}, and using $\mathbb Pv=v$ and the
density of $\mathscr D_\sigma(\Omega)$ in $L^2_\sigma(\Omega)$ from
Lemma~\ref{lem:density}, we obtain
\begin{equation}
\partial_tu_{N,Q}=-\mathbb P\operatorname{div}S_N^hS_Q^v(u\otimes u)+\nu\,\partial_1^2u_{N,Q}
\qquad\text{in }\mathcal D'(0,T;L^2_\sigma(\Omega)),
\label{eq:ptwNQ}
\end{equation}
with right-hand side in $L^2(0,T;L^2_\sigma(\Omega))$ by \eqref{eq:wAreg}. Hence
$u_{N,Q}\in H^1(0,T;L^2_\sigma(\Omega))$, so $t\mapsto\|u_{N,Q}(t)\|_{L^2}^2$ is
absolutely continuous with derivative $2(\partial_tu_{N,Q},u_{N,Q})_{L^2}$, and the
inner product of \eqref{eq:ptwNQ} with $u_{N,Q}$, using $\mathbb Pu_{N,Q}=u_{N,Q}$ and
\eqref{eq:wAreg}, yields
\begin{equation}
\frac12\frac{d}{dt}\|u_{N,Q}\|_{L^2}^2+\nu\|\partial_1u_{N,Q}\|_{L^2}^2
=\int_\Omega S_N^hS_Q^v(u\otimes u):\nabla u_{N,Q}\,dx
\qquad\text{for a.e.\ }t\in(0,T).
\label{eq:energyNQ}
\end{equation}

By \eqref{eq:PiNQ}, $S_N^hS_Q^v(u\otimes u)=u_{N,Q}\otimes u_{N,Q}+R_{N,Q}(u)$, and,
since $u_{N,Q}(t)\in H^1_\sigma(\Omega)$, the integral of
$(u_{N,Q}\otimes u_{N,Q}):\nabla u_{N,Q}$ over $\Omega$ vanishes, by integration by
parts on $\mathscr D_\sigma(\Omega)$ and then Lemma~\ref{lem:density} and
\eqref{eq:ladyzh}. So the right-hand side of \eqref{eq:energyNQ} equals $\Pi_{N,Q}$,
which lies in $L^1(0,T)$ by \eqref{eq:hypL4}, \eqref{eq:RNQL2} and \eqref{eq:wAreg}.
Finally, $t\mapsto S_N^hS_Q^vu(t)$ is weakly continuous on $[0,T]$ by
Remark~\ref{rem:weakcont}, since $S_N^hS_Q^v$ is bounded on $L^2$. It agrees almost
everywhere with the continuous representative of
$u_{N,Q}\in H^1(0,T;L^2_\sigma(\Omega))$, hence for every $t\in[0,T]$. Integrating
\eqref{eq:energyNQ} over $(0,t)$ proves \eqref{eq:balanceNQ}.
\end{proof}

\end{appendices}


\begin{thebibliography}{99}

\bibitem{BCD} Bahouri, H., Chemin, J.-Y., Danchin, R.: Fourier Analysis and Nonlinear
Partial Differential Equations. Grundlehren der mathematischen Wissenschaften, vol. 343.
Springer, Heidelberg (2011)
\url{https://doi.org/10.1007/978-3-642-16830-7}

\bibitem{BerselliChiodaroli2020} Berselli, L.C., Chiodaroli, E.: On the energy equality for
the 3D Navier--Stokes equations. Nonlinear Anal. 192, 111704 (2020)
\url{https://doi.org/10.1016/j.na.2019.111704}

\bibitem{BDSV2019} Buckmaster, T., De Lellis, C., Sz\'ekelyhidi, L., Jr., Vicol, V.:
Onsager's conjecture for admissible weak solutions. Commun. Pure Appl. Math. 72, 229--274
(2019)
\url{https://doi.org/10.1002/cpa.21781}

\bibitem{CDGG2000} Chemin, J.-Y., Desjardins, B., Gallagher, I., Grenier, E.: Fluids with
anisotropic viscosity. M2AN Math. Model. Numer. Anal. 34, 315--335 (2000)
\url{https://doi.org/10.1051/m2an:2000143}

\bibitem{CDGGbook} Chemin, J.-Y., Desjardins, B., Gallagher, I., Grenier, E.: Mathematical
Geophysics: An Introduction to Rotating Fluids and the Navier--Stokes Equations. Oxford
Lecture Series in Mathematics and Its Applications, vol. 32. Oxford University Press,
Oxford (2006)
\url{https://doi.org/10.1093/oso/9780198571339.001.0001}

\bibitem{CCFS2008} Cheskidov, A., Constantin, P., Friedlander, S., Shvydkoy, R.: Energy
conservation and Onsager's conjecture for the Euler equations. Nonlinearity 21, 1233--1252
(2008)
\url{https://doi.org/10.1088/0951-7715/21/6/005}

\bibitem{CheskidovLuo2020} Cheskidov, A., Luo, X.: Energy equality for the Navier--Stokes
equations in weak-in-time Onsager spaces. Nonlinearity 33, 1388--1403 (2020)
\url{https://doi.org/10.1088/1361-6544/ab60d3}

\bibitem{CheskidovLuo2023} Cheskidov, A., Luo, X.: $L^2$-critical nonuniqueness for the 2D
Navier--Stokes equations. Ann. PDE 9, Paper No. 13 (2023)
\url{https://doi.org/10.1007/s40818-023-00154-9}

\bibitem{CET1994} Constantin, P., E, W., Titi, E.S.: Onsager's conjecture on the energy
conservation for solutions of Euler's equation. Commun. Math. Phys. 165, 207--209 (1994)
\url{https://doi.org/10.1007/BF02099744}

\bibitem{DemmelWiedemann2026} Demmel, J., Wiedemann, E.: Energy rigidity and weak-strong
uniqueness for the 2D anisotropic Navier--Stokes equations. arXiv:2608.19931v1 (2026)
\url{https://doi.org/10.48550/arXiv.2608.19931}

\bibitem{DuchonRobert2000} Duchon, J., Robert, R.: Inertial energy dissipation for weak
solutions of incompressible Euler and Navier--Stokes equations. Nonlinearity 13, 249--255
(2000)
\url{https://doi.org/10.1088/0951-7715/13/1/312}

\bibitem{Eyink1994} Eyink, G.L.: Energy dissipation without viscosity in ideal
hydrodynamics I. Fourier analysis and local energy transfer. Phys. D 78, 222--240 (1994)
\url{https://doi.org/10.1016/0167-2789(94)90117-1}

\bibitem{FLRT2000} Furioli, G., Lemari\'e-Rieusset, P.-G., Terraneo, E.: Unicit\'e dans
$L^3(\mathbb R^3)$ et d'autres espaces fonctionnels limites pour Navier--Stokes. Rev. Mat.
Iberoam. 16, 605--667 (2000)
\url{https://doi.org/10.4171/RMI/286}

\bibitem{Galdi2019} Galdi, G.P.: On the energy equality for distributional solutions to
Navier--Stokes equations. Proc. Am. Math. Soc. 147, 785--792 (2019)
\url{https://doi.org/10.1090/proc/14256}

\bibitem{GiriRadu2024} Giri, V., Radu, R.-O.: The Onsager conjecture in 2D: a Newton--Nash
iteration. Invent. Math. 238, 691--768 (2024)
\url{https://doi.org/10.1007/s00222-024-01291-z}

\bibitem{Hopf1951} Hopf, E.: \"Uber die Anfangswertaufgabe f\"ur die hydrodynamischen
Grundgleichungen. Math. Nachr. 4, 213--231 (1951)
\url{https://doi.org/10.1002/mana.3210040121}

\bibitem{Isett2018} Isett, P.: A proof of Onsager's conjecture. Ann. Math. 188, 871--963
(2018)
\url{https://doi.org/10.4007/annals.2018.188.3.4}

\bibitem{Ladyzhenskaya1969} Ladyzhenskaya, O.A.: The Mathematical Theory of Viscous
Incompressible Flow, 2nd edn. Gordon and Breach, New York (1969)

\bibitem{Leray1934} Leray, J.: Sur le mouvement d'un liquide visqueux emplissant l'espace.
Acta Math. 63, 193--248 (1934)
\url{https://doi.org/10.1007/BF02547354}

\bibitem{LZZ} Liang, S., Zhang, P., Zhu, R.: Deterministic and stochastic 2D Navier--Stokes
equations with anisotropic viscosity. J. Differ. Equ. 275, 473--508 (2021)
\url{https://doi.org/10.1016/j.jde.2020.11.028}

\bibitem{Lions1960} Lions, J.-L.: Sur la r\'egularit\'e et l'unicit\'e des solutions
turbulentes des \'equations de Navier--Stokes. Rend. Semin. Mat. Univ. Padova 30, 16--23
(1960)
\url{https://www.numdam.org/item/RSMUP_1960__30__16_0/}

\bibitem{Onsager1949} Onsager, L.: Statistical hydrodynamics. Nuovo Cimento (9) 6, Suppl.
2, 279--287 (1949)
\url{https://doi.org/10.1007/BF02780991}

\bibitem{Paicu2005} Paicu, M.: \'Equation anisotrope de Navier--Stokes dans des espaces
critiques. Rev. Mat. Iberoam. 21, 179--235 (2005)
\url{https://doi.org/10.4171/RMI/420}

\bibitem{Shinbrot1974} Shinbrot, M.: The energy equation for the Navier--Stokes system.
SIAM J. Math. Anal. 5, 948--954 (1974)
\url{https://doi.org/10.1137/0505092}

\end{thebibliography}
\end{document}